\documentclass[11pt, a4paper]{article}
\usepackage{amsmath,amssymb,amsfonts,graphicx,amsthm}

\usepackage{amsthm}
\usepackage[utf8]{inputenc}
\usepackage{graphicx}
\usepackage{authblk}
\usepackage[T1]{fontenc}

\newtheorem{thm}{{\bf  Theorem}}
\newtheorem{preconj}{{\bf  Conjecture}}

\input epsf
\begin{document}

\title{\large \bf A Note on Induced Path Decomposition of Graphs
\thanks {2010{ \it Mathematics Subject Classification}: 05C78
}}

\author{{\normalsize
{\sc S. Akbari${}^{\mathsf{a}}$},\,
{\sc H.R. Maimani${}^{\mathsf{b}}$},\,
{\sc A. Seify${}^{\mathsf{b}} {}^{\mathsf{,c}}$}}
\vspace{3mm}
\\{\footnotesize{${}^{\mathsf{a}}$\it Department of
Mathematical Sciences, Sharif University of Technology, Tehran,
Iran}}
{\footnotesize{}}\\{\footnotesize{${}^{\mathsf{b}}$\it Mathematics Section, Department of
Sciences, Shahid Rajaee Teacher Training University, Tehran,
Iran}}
{\footnotesize{}}\\{\footnotesize{${}^{\mathsf{c}}$\it School of
Mathematics, Institute for Studies in Theoretical Physics and
Mathematics,}}{\footnotesize{}}\\{\footnotesize{${}^{\mathsf{}}$\it
P.O. Box 19395-5746,
Tehran, Iran.}}
\thanks{{\it E-mail addresses}: $\mathsf{s\_akbari@sharif.edu}$,
$\mathsf{maimani@ipm.ir}$ and $\mathsf{abbas.seify@gmail.com}$.} }

\date{}

\maketitle

\begin{abstract}
{\small \noindent 
Let $G$ be a graph of order $n$. The path decomposition of $G$ is a set of disjoint paths, say $\mathcal{P}$, which cover all vertices of $G$. If all paths are induced paths in $G$, then we say $\mathcal{P}$ is an induced path decomposition of $G$. Moreover, if every path is of order at least 2, then we say that $G$ has an IPD. In this paper, we prove that every connected $r$-regular graph which is not complete graph of odd order admits an IPD. Also we show that every connected bipartite cubic graph of order $n$ admits an IPD of size at most $\frac{n}{3}$. We classify all connected claw-free graphs which admit an IPD.
}

\end{abstract}

\section{Introduction}
Let $G$ be a simple graph. The vertex set and the edge set of $G$ are denoted by $V(G)$ and $E(G)$, respectively. Also, $|V(G)|$ and $|E(G)|$ are called {\it order} and {\it size} of $G$. The {\it degree} of $v$ is denoted by $d(v)$. If $d(v)=r$, for every $v \in V(G)$, then $G$ is called an $r$-{\it regular} graph. A subgraph $H \subseteq G$ is called {\it induced} if for every two vertices $u,v \in V(H)$, $uv \in E(H)$ if and only if $uv \in E(G)$. Also, $H$ is called {\it spanning} if $V(H)=V(G)$. For two positive integers $m$ and $n$, $K_{m,n}$ denote the complete bipartite graph with part sizes $m$ and $n$, respectively. The cycle and the path of order $n$, denoted by $C_n$ and $P_n$, respectively. A graph $G$ of order $n$ is called {\it Hamiltonian} if it contains a cycle of order $n$. A {\it factor} $F$ is a spanning subgraph of $G$. If $d_F(v) =r$, for every $v \in V(G)$, then $F$ is called an {\it $r$-factor}. A $\{C_n: n \geq k\}$-factor is a factor in which every component is a cycle with at least $k$ vertices. Let $G$ and $H$ be two graphs. We say that $G$ is {\it $H$-free}, if $G$ contains no induced subgraph isomorphic to $H$. A graph $G$ is called {\it claw-free} if $G$ is $K_{1,3}$-free.
A connected graph $G$ is called {\it $2$-(edge)connected} if by eliminating every (edge) vertex, $G$ remains connected. A {\it block} is a maximal 2-connected subgraph of $G$. Every connected graph $G$ has a block decomposition such as $B_1, \ldots, B_b$ in which each $B_i$ is a block of $G$, $\bigcup_{i=1}^{b}E(B_i)= E(G)$ . The block $B$ is called a {\it leaf block}, if it contains at most one cut vertex of $G$. 
\\
A {\it path decomposition} of $G$ is a set of vertex disjoint paths, say $\mathcal{P}=\{P_1, \ldots, P_t\}$, such that $V(G)= \bigcup_{i=1}^{t} V(P_i)$. Moreover, if every $P_i$ is an induced path, then $\mathcal{P}$ is called an {\it induced path decomposition} of $G$. If all paths are induced of order at least two, then we say that $G$ has an {\it IPD}. We denote the smallest size of an IPD for $G$ by $\rho(G)$. Induced path number was studied by several authors, for instance see \cite{broere}, \cite{chartrand}, \cite{Hatting1}, \cite{Hatting2}, \cite{Hatting3}. We consider the induced path decomposition of graphs. We prove some results concerning the existence of an IPD. We use the following interesting and strong theorem which proved by Egawa et al. to prove the existence of an IPD for an $r$-regular graph except complete graph of odd order.

\begin{thm}\label{EKK}{\rm \cite{EKK}}
Let $G$ be a graph and $n \geq 2$ be an integer. Then the following are equivalent:
\\
1. $V(G)$ can be partitioned into $V_1, \ldots, V_t$, where each $V_i$ induces one of $K_{1,1}, K_{1,2}, \ldots, K_{1,n}$.
\\
2. For every $S \subseteq V(G)$, $G \setminus S$ has at most $n|S|$ components with the property that each of their block is an odd order complete graph.
\end{thm}

Also, we show that every Hamiltonian graph $G \neq K_{2n+1}$,  admits an IPD. Moreover, in this paper we prove that for every connected cubic graph $G$ of order $n$, $\rho(G) \leq \frac{n}{3}$.
\section{Results}
We study the existence of IPD for a graph. Note that if $G$ has an IPD, then every path in IPD can be decomposed into paths of length 1 and 2. So, the existence of IPD is equivalent to the existence of an induced $\{P_2, P_3\}$-decomposition. Since $P_2=K_{1,1}$ and $P_3=K_{1,2}$, we can use Theorem \ref{EKK}. Now, we prove some results concerning the existence of IPD in graphs.

\begin{thm}\label{claw-free}
Let $G$ be a connected claw-free graph. Then $G$ has an IPD if and only if $G$ has at least one block which is not a complete graph of odd order
\end{thm}
 
\begin{proof}
By contrary, suppose that $G$ has no IPD. By Theorem \ref{EKK}, there exists $S \subseteq V(G)$ such that $G \setminus S$ has at least $2|S|+1$ components with the property that any block of each component is a complete graph of odd order. We call these components bad components. Note that if $S= \varnothing$, then there is nothing to prove. So, we may assume that $S \neq \varnothing$. So, $G \setminus S$ has at least three bad components and
there exists at least one edge between every bad component and $S$. Now, by pigeon hole principle there exists $s \in S$ such that it has at least 3 edges to three distinct bad components. This yields that $G$ has a claw, a contradiction. 
\\
For the other side, let every block of $G$ be a complete graph of odd order. We show that $G$ has no IPD. If $G$ is a complete graph of odd order, then clearly $G$ has no IPD. Thus, assume that $G$ has at least two blocks. Let $B$ be a leaf block of $G$ and $B$ contains a unique cut vertex $v$ of $G$. By induction hypothesis, $G \setminus (B \setminus \{v\})$ has no IPD. Now, if $G$ has an IPD, then one of the induced paths in IPD should be $uvw$, where $u \in V(B) \setminus \{v\}$ and $w \notin V(B)$. This implies that $B \setminus \{u,v\}$  is a complete graph of odd order which has an IPD, a contradiction.  
\end{proof} 

\begin{thm}\label{regular}
Let $G$ be a connected $r$-regular graph which is not a complete graph of odd order. Then $G$ has an IPD. 
\end{thm}

\begin{proof}
For $r \leq 2$ the assertion is clear. Thus assume that $r \geq 3$. First, in Theorem \ref{EKK} let $n=2$. Suppose that $S$ is a subset of $V(G)$ of size $s$. We call a graph bad if each of the blocks is an odd complete graph. Suppose that $G \setminus S$ contains at least $2s+1$ bad components, say $H_1, \ldots, H_{2s+1}$. Let $B$ be a leaf block of $H_1$ or $B=H_1$, when $H_1$ is 2-connected. Since $G$ is $r$-regular, the number of edges between $V(H_1)$ and $S$ is at least $r-1$, to see this, suppose that $B=K_p$, where $1 \leq p \leq r$. So, the number of edges between $V(B)$ and $S$ is at least $(p-1)(r-p+1) \geq r-1$. Thus, the number of edges between $\bigcup_{i=1}^{i=2s+1} V(H_i)$ and $S$ is at least $(2s+1)(r-1)$. Now, by pigeon hole principle, there exists a vertex $v \in S$ such that $d(v) \geq 2(r-1)+1$, a contradiction. Therefore, by Theorem \ref{EKK} every $r$-regular graph has an IPD.
\end{proof}

\begin{thm}\label{hamiltonian}
Let $G$ be a Hamiltonian graph which is not a complete graph of odd order. Then $G$ has an IPD. 
\end{thm}

\begin{proof}
Let $C: v_1, \ldots, v_n$ be a Hamiltonian cycle in $G$. First, notice that if $n$ is even, then we are done. So, we may assume that $n=2t+1$, for some positive integer $t$. Note that if there exists some positive integer $i$, $1 \leq i \leq 2t+1$ such that $v_{i}v_{i+2} \notin E(G)$, then $V(G)$ can be decomposed into $P=v_{i}v_{i+1}v_{i+2}$ and $t-1$ paths of order 2 and so $G$ has an IPD. So, we may assume that $v_{i}v_{i+2} \in E(G)$, (mod $2t+1$), for $i=1,2, \ldots, 2t+1$. Now, note that if $v_1v_5 \notin E(G)$, then $\mathcal{P}= \{v_1v_3v_5, v_2v_4, v_6v_7, v_8v_9, \ldots, v_{2t}v_{2t+1}\}$ is an IPD for $G$. So, we may assume that $v_1v_5 \in E(G)$. Also, if $v_1v_7 \notin E(G)$, then $\mathcal{P}= \{v_1v_5v_7, v_2v_3, v_4v_6, v_8v_9, \ldots, v_{2t}v_{2t+1}\}$ is an IPD for $G$. By this method, one can see that $v_1v_{2i+1} \in E(G)$, for $i=1, \ldots, t$.
\\
Also, note that if $v_1v_4 \notin E(G)$, then $\mathcal{P}=\{v_1v_3v_4, v_5v_6, \ldots, v_{2t-1}v_{2t}, v_{2t+1}v_{2}\}$ is an IPD for $G$. Also if $v_1v_6 \notin E(G)$, then $\mathcal{P}=\{v_1v_5v_6, v_3v_4, \ldots, v_{2t-1}v_{2t}, v_{2t+1}v_{2}\}$ is an IPD for $G$. Similarly, one can see that $v_1v_{2i} \in E(G)$, for $i=1, \ldots, t$. This implies that $v_1$ is adjacent to all vertices of graph. Since $v_1$ is an arbitrary vertex, $G = K_{2t+1}$, a contradiction. This completes the proof.
\end{proof}

Now, we would like to study the induced path number of cubic graphs. Before proving our last result, we need the following theorem.

\begin{thm}\label{c6factor}{\rm \cite{kanoc6}}
Every connected cubic bipartite graph has a $\{C_n: n \geq 6 \}$-factor. 
\end{thm}

Using Theorem \ref{c6factor}, we can prove the following.

\begin{thm}
Let $G$ be a connected bipartite cubic graph of order $n$. Then $\rho(G) \leq \frac{n}{3}$.
\end{thm}

\begin{proof}
By Theorem \ref{c6factor}, $G$ has a $\{C_n: n \geq 6 \}$-factor, say $F$. Let $C_t=v_1, \ldots, v_t,v_1$ be an arbitrary cycle in $F$. We show that the induced subgraph on $V(C_t)$ can be decomposed into at most $\frac{t}{3}$ induced paths and hence we are done. 
\\
First suppose that $t=3k$, for some positive integer $k$. Since $G$ is triangle-free, $C_t$ can be decomposed into $k$ induced paths of length 3. Now, let $t=3k+1$. Note that if $v_1v_{4} \notin E(G)$, then $\mathcal{P}=\{v_1v_{2}v_{3}v_{4}, v_{5}v_{6}v_{7}, \ldots, v_{t-2}v_{t-1}v_{t}\}$ is an IPD for $G$. Thus $v_{1}v_{4} \in E(G)$. Similarly, $v_1v_{t-2} \in E(G)$. This implies that $d(v_1) \geq 4$, a contradiction.
\\
Finally, assume that $t=3k+2$. Note that if $v_1v_4, v_5v_8 \notin E(G)$, then $\mathcal{P}=\{v_1v_2v_3v_4,\break  v_5v_6v_7v_8, v_9v_{10}v_{11},  \ldots, v_{t-2}v_{t-1}v_t\}$ is the desired IPD. So, assume that $v_1v_4 \in E(G)$. Note that if $v_{t-1}v_2 \in E(G)$, then 
$\mathcal{P}=\{v_{t-2}v_{t-1}v_{t}v_1, v_2v_3v_4v_5, v_6v_7v_8,  \ldots, v_{t-5}v_{t-4}v_{t-3}\}$ is the desired IPD. So, suppose that $v_{t-1}v_2 \notin E(G)$. Now, since $G$ has no $C_5$, by considering the paths $v_{t-2}v_{t-1}v_tv_1v_2$ and $v_{3i}v_{3i+1}v_{3i+2}$, for $i=1, \ldots, k-1$, we are done. This completes the proof.
\end{proof}

\noindent
\textbf{Acknowledgements.} The authors are deeply grateful to Kenta Ozeki for introducing Theorem \ref{EKK}.


\begin{thebibliography}{mm} 

\bibitem{broere} I. Broere, E. Jonck, G.S. Domke, The induced path number of the complements of some graphs, {\it Australas J. Combin.} Volume 33(2005), pp 15-32.  

\bibitem{chartrand} G. Chartrand, J. McCanna, N. Sherwani, J. Hashmi, M. Hossain, The induced path number of bipartite graphs, {\it Ars Combinatoria}, 37(1994), pp 191-208. 

\bibitem{EKK} Y. Egawa, M. Kano, A.K. Kelmans, Star Partition of Graphs, {\it J. Graph Theory}, Vol. 25, Issue 3 (1997), pp 185-190. 

\bibitem{Hatting1} J.H. Hattingh, O.A. Saleh, L.C. Van der Merwe, T.J. Walters, Nordhaus–Gaddum results for the sum of the induced path number of a graph and its complement. {\it Acta Math. Sin.} 28, pp 2365-2372 (2012).

\bibitem{Hatting2} J.H. Hattingh, O.A. Saleh, L.C. Van der Merwe, T.J. Walters, Product Nordhaus–Gaddum-type
results for the induced path number with relative complements in $K_n$ or $K_{n,n}$. {\it Util. Math.} 94, pp 275-285 (2014).

\bibitem{Hatting3} J. H. Hattingh, O. A. Saleh, L. C. van der Merwe, T. J. Walters, Nordhaus–Gaddum Results for the Induced Path Number of a Graph When Neither the Graph Nor Its Complement Contains Isolates, {\it Graphs Combin.} (2016) 32, pp 987-996.

\bibitem{kanoc6} M. Kano, C. Lee and K. Suzuki, Path factors and cycle factors of cubic bipartite graphs, {\it Discuss. Math. Graph Theory.} 28(2008) 551-556.






\end{thebibliography}
\end{document}